\documentclass[12pt]{article}

\usepackage{amsmath}
\usepackage{amsfonts}
\usepackage{amsthm}

\usepackage[utf8]{inputenc}

\usepackage{color}

\usepackage{multirow}

\usepackage{multicol}

\usepackage{booktabs}

\newcommand{\ma}{\textit{Mathematica$^{\small{\circledR}}$}}

\date{}

\newtheorem{theo}{Theorem}
\newtheorem{lemma}{Lemma}

\newtheorem{prop}{Proposition}
\newtheorem{nota}{Remark}
\newtheorem{ejemplo}{Example}

\begin{document}

\title{Asymptotics for some $q$-hypergeometric polynomials}

\author{Juan F. Ma\~{n}as--Ma\~{n}as$^a$, Juan J. Moreno--Balc\'{a}zar$^{a,b}$.}

\maketitle

{\scriptsize
\noindent
$^a$Departamento de Matem\'{a}ticas, Universidad de Almer\'{\i}a, Spain.\\
$^b$Instituto Carlos I de F\'{\i}sica Te\'{o}rica y Computacional, Spain.\\

\noindent \textit{E-mail addresses:}  (jmm939@ual.es) J. F. Ma\~{n}as--Ma\~{n}as, (balcazar@ual.es) J. J. Moreno--Balc\'{a}zar.
}

\begin{abstract}
We tackle the study of a type of local asymptotics, known as Mehler--Heine asymptotics,  for some $q$--hypergeometric polynomials. Some consequences  about the asymptotic behavior of the zeros  of these polynomials are discussed. We illustrate the results with numerical examples.
\end{abstract}



\section{Introduction}\label{introducction}

The basic  $q$-hypergeometric function $ _r\phi_s$ is defined by the series (see, for example, \cite{NIST-q-Hyper} or \cite[f. (1.10.1)]{Koekoek-book-hyper})

\begin{equation}\label{q-basic}
_r\phi_{s}\left(\begin{array}{l}
a_{1}, \ldots, a_{r} \\
b_{1}, \ldots, b_{s}
\end{array} ; q, z\right)
=\sum_{k=0}^{\infty} \frac{\left(a_{1},\dots,a_{r} ; q\right)_{k}}{\left(b_{1},\dots,b_{s} ; q\right)_{k}}(-1)^{(1+s-r) k} q^{(1+s-r)\binom{k}{2}} \frac{z^{k}}{(q ; q)_{k}},
\end{equation}
where $\left(a_{1},\dots,a_{r} ; q\right)_{k}=\left(a_{1}; q\right)_{k}\left(a_{2}; q\right)_{k}\cdots\left(a_{r} ; q\right)_{k}$.
For our purposes we assume throughout the paper $0<q<1.$ The expressions $ \left(a_{j}; q\right)_{k}$ and $\left(b_{j}; q\right)_{k}$ denote  the $q$-analogues of the Pochhammer symbol, i.e., given a  complex number  $a$
\begin{equation}\label{q-Pochhammer}
(a ; q)_{0}:=1 \text { and }(a ; q)_{n}=\prod_{k=1}^{n}\left(1-a q^{k-1}\right)=\prod_{k=0}^{n-1}\left(1-a q^{k}\right), \qquad n \geq 1,
\end{equation}
with
\begin{equation}\label{q-Pochhammer-infty}
(a ; q)_{\infty}=\prod_{k=0}^{\infty}\left(1-a q^{k}\right).
\end{equation}

Obviously, the series (\ref{q-basic}) is well--defined when the quantities, known as $q$--shifted factorials or $q$--Pochhammer symbols,  $\left(b_{j}; q\right)_{k}\neq 0,$ for $j=0, \ldots, s.$ It is well known that the radius of convergence $\rho$ of the $q-$hypergeometric functions (\ref{q-basic}) is given by (see for example \cite[p. 15]{Koekoek-book-hyper})
\begin{equation*}
\rho=\left\{
       \begin{array}{ll}
         \infty, & \hbox{if $r<s+1$;} \\
         1, & \hbox{if $r=s+1$;} \\
         0, & \hbox{if $r>s+1$.}
       \end{array}
     \right.
\end{equation*}
In particular, for our interest $ _s\phi_s$ is always convergent. In this paper we consider  $r=s$ because it is the context where we can establish our main goal (see Theorem  \ref{th-MH}).

The $q$--series (1) is the analogous series in the framework of the $q$--analysis to the hypergeometric function given by (see, for example,  \cite{Andrews-et-1999} or \cite{NIST-hyper})
\begin{equation*}
{ }_{r} F_{s}\left(\begin{array}{l}
a_{1}, \ldots, a_{r} \\
b_{1}, \ldots, b_{s}
\end{array} ; z\right)=\sum_{k=0}^{\infty} \frac{\left(a_{1}\right)_{k} \cdots\left(a_{r}\right)_{k}}{\left(b_{1}\right)_{k} \cdots\left(b_{r}\right)_{k}} \frac{z^{k}}{k !}.
\end{equation*}
They are connected by the limit relation
\begin{equation*}\label{relationH-qH}
\lim _{q \rightarrow 1}\ _r\phi_{s}\left(\begin{array}{l}
q^{a_{1}}, \ldots, q^{a_{r}} \\
q^{b_{1}}, \ldots, q^{b_{s}}
\end{array} ; q,(q-1)^{1+s-r}z\right)={ }_{r} F_{s}\left(\begin{array}{l}
a_{1}, \ldots, a_{r} \\
b_{1}, \ldots, b_{s}
\end{array} ; z\right).
\end{equation*}
In particular, when $r=s$, we get
\begin{equation}\label{relationH-qH-r=s}
\lim _{q \rightarrow 1}\ _s\phi_{s}\left(\begin{array}{l}
q^{a_{1}}, \ldots, q^{a_{s}} \\
q^{b_{1}}, \ldots, q^{b_{s}}
\end{array} ; q,(q-1)z\right)={ }_{s} F_{s}\left(\begin{array}{l}
a_{1}, \ldots, a_{s} \\
b_{1}, \ldots, b_{s}
\end{array} ; z\right).
\end{equation}

When one of the  parameters $a_j $ in (\ref{q-basic}) is equal to $q^{-n}$,
where $n$ is a nonnegative integer, the basic $q$--hypergeometric function is a polynomial of degree at most $n $ in the variable $z$. Thus, our objective is to obtain a type of asymptotics for these $q$--polynomials. Concretely, by scaling  adequately these polynomials we intend to get a limit relation between them and a $q$--analogue of the Bessel function of the first kind. In the framework of orthogonal polynomials, this type of asymptotics is known as Mehler--Heine asymptotics (also as Mehler--Heine formula). Originally, this type of local asymptotics was introduced for a special case of orthogonal polynomials (OP), Legendre polynomials,  by the German mathematicians H. E. Heine and G. F. Mehler in the 19th century. Later, it was extended to the families of classical OP: Jacobi, Laguerre, Hermite (see, for example, \cite{sz}). More recently, these formulae were obtained for other families of polynomials such as OP in the Nevai's class \cite{aptekarev-1993}, discrete OP \cite{dom}, generalized Freud polynomials \cite{ambpr}, multiple OP \cite{takata}, \cite{wva} or Sobolev OP (in this context there are many papers in this century, being \cite{marmb} one of the first), among others.

These formulae have a nice consequence about the scaled zeros of the polynomials, i.e. using the well--known Hurwitz's theorem we can establish a limit relation between these scaled zeros and the ones of a Bessel function of the first kind. Thus, in \cite{Bracciali-JJMB-2015} the authors, starting from a Mehler--Heine formula for hypergeometric polynomials, make a study of the zeros of these polynomials. Notice that in this case the polynomials are not necessarily orthogonal. In this way, we are looking for  a similar result in the context of the $q$--analysis.   In fact,  we can find several nice works where the authors study Placherel--Rotach asymptotics  for  basic hypergeometric polynomials (see, for example \cite{Ismail-IMRN-2005}--\cite{Ismail-Zhang-2018}, \cite{Li-Wong-2013}, \cite{Wang-Wong-2010} and \cite{Zhang-2008b}--\cite{Zhang-2012}). In this type of asymptotics, most authors scale the variable $z$ using a divergent sequence, this is, $z \to a_n z$ with $a_n\to \infty$ when $n\to +\infty$. In these works the authors usually obtain relations between the basic hypergeometric polynomials and the $q$-Airy function.
We would like to highlight the work \cite{Li-Wong-2013} where the authors obtain a relation between the Stieltjes-Wigert orthogonal  polynomials $S_n(z;q)$ and the $q$-Airy function $A_q(z)$ using a scaling of the  variable satisfying $a_n\to 0$ when $n\to +\infty$.  In fact, the authors prove that (see \cite[Th. 1]{Li-Wong-2013})
$$S_n(z;q)=\frac{1}{(q;q)_n}\left(A_q(z)+r_n(z)\right),$$
where $z=uq^{-nt}$ with $-\infty<t<2$, $u\in\mathbb{C}\backslash\{0\}$ and  being $r_n(z)$ a remainder function (for more details see \cite[Th. 1]{Li-Wong-2013}). Clearly, $uq^{-nt}\to 0$ when $t<0$ and $n\to +\infty$.  To prove their results the authors use a symmetry property of the Stieltjes-Wigert orthogonal polynomials given by
$$S_n(z;q)=(-zq^n)S_n\left(\frac{1}{zq^{2n}};q\right).$$

In our case, we will show that the variable is scaled $z\to a_nz$ by a sequence $a_n$ satisfying  $a_n\to 0$ when $n\to +\infty$. So,  we can establish asymptotic relations between these $q$--polynomials  and a $q$--Bessel function.   Thus, the novelty of our approach is to extend the classical Mehler--Heine formulae to these $q$--hypergeometric polynomials.

Now, we establish the notation that we will use in the following sections  and we will show our main result. We denote by $[z]_q$ the well--known $q$-number given by
\begin{equation}\label{q-Number}[z]_q=\frac{1-q^{z}}{1-q},\end{equation}
for $0<q<1$, it is easy to prove that
\begin{equation} \label{limqnum} \lim _{q \rightarrow 1}[z]_{q}=z.\end{equation}
In addition, we will use the $q$-Gamma function given by (see, for example,
\cite[f. (10.3.3)]{Andrews-et-1999})
\begin{equation}\label{q-gamma}
\Gamma_{q}(z)=\frac{\left(q; q\right)_{\infty}}{\left(q^{z} ; q\right)_{\infty}}(1-q)^{1-z}, \qquad 0<q<1.
\end{equation}
This function is a $q$-analogue of the Gamma function. It is a meromorphic function, without zeros and with poles in $z=-n\pm2\pi i k /\log(q)$ where $n$ and $k$ are
nonnegative integers. This function verifies the relation (see \cite[Pag. 495]{Andrews-et-1999})
\begin{equation}\label{q-gamma-limit}
\lim _{q \rightarrow 1} \Gamma_{q}(z)=\Gamma(z),
\end{equation}
and satisfies
\begin{equation}\label{q-gamma-recurrence}
\Gamma_{q}(z+1)=[z]_{q} \Gamma_{q}(z) \qquad \operatorname{with} \qquad \Gamma_{q}(1)=1.
\end{equation}
An important role in this paper is played by  the  $q$--Bessel function $J_{\alpha}^{(2)}(z ; q)$  given by (see, for example, \cite[f. (1.14.8)]{Koekoek-book-hyper})
\begin{equation}\label{bessel-q-2}
J_{\alpha}^{(2)}(z ; q)=\frac{\left(q^{\alpha+1} ; q\right)_{\infty}}{(q ; q)_{\infty}}\left(\frac{z}{2}\right)^{\alpha} \ _0\phi_{1}\left(\begin{array}{c}
                                                                                                                                               - \\
                                                                                                                                               q^{\alpha+1}
                                                                                                                                           \end{array}
 ; q, \frac{-q^{\alpha+1}z^{2}}{4}\right),
\end{equation}
which is an extension of the Bessel function of the first kind $J_{\alpha}(z)$, i.e.
\begin{equation}\label{q-bessel-limit}
\lim _{q \rightarrow 1}J_{\alpha}^{(2)}((1-q)z ; q)=J_{\alpha}(z).
\end{equation}
With this notation, we will prove in the  Theorem \ref{th-MH} that, for $s\geq2,$
\begin{align*}\nonumber
&\lim_{n\to+\infty} \ _s\phi_{s}\left(\begin{array}{l}
q^{-n},q^{\mathfrak{a}_sn+\mathfrak{b}_s} \\
q^{\alpha},q^{\mathfrak{c}_sn+\mathfrak{d}_s}
\end{array} ; q, \frac{q^{n+\alpha}[n]_{q^{\mathfrak{c}_s}}}{[n]_q[n]_{q^{\mathfrak{a}_s}}}(q-1)z\right)\\
&= \left( \frac{[\mathfrak{a}_s]_q}{[\mathfrak{c}_s]_q}z\right)^{\frac{1-\alpha}{2}} \Gamma_{q}(\alpha) J_{\alpha-1}^{(2)}\left(2 (1-q) \sqrt{\frac{[\mathfrak{a}_s]_q}{[\mathfrak{c}_s]_q}z} ; q\right),
\end{align*}
where
\begin{align}\label{notation1}
q^{\mathfrak{a}_sn+\mathfrak{b}_s}&=q^{a_1n+b_1},q^{a_2n+b_2},\cdots,q^{a_{s-1}n+b_{s-1}},\\ \label{notation2}
[n]_{q^{\mathfrak{a}_s}}&=[n]_{q^{a_1}}[n]_{q^{a_2}}\cdots[n]_{q^{a_{s-1}}},\\ \label{notation3}
[\mathfrak{a}_s]_q&=[a_1]_q[a_2]_q\cdots[a_{s-1}]_q.
\end{align}
The above result is also true when $s=1$,  getting in this case the following relation:
\begin{equation*}
\lim_{n\to+\infty} \ _1\phi_{1}\left(\begin{array}{l}
q^{-n} \\
q^{\alpha},
\end{array} ; q, \frac{q^{n+\alpha}}{[n]_q}(q-1)z\right)= z^{\frac{1-\alpha}{2}} \Gamma_{q}(\alpha) J_{\alpha-1}^{(2)}\left(2 (1-q) \sqrt{z} ; q\right).
\end{equation*}

We also discuss the case $r-1\leq s$ in Proposition \ref{othcas}. In addition,  one of the referees asked us what result would be obtained when we take the scale $z\to q^nz$. These results are shown in Propositions \ref{qn} and \ref{pro-qnrs}. They are very nice but, as far as we know, they cannot be used to deduce the result obtained in \cite{Bracciali-JJMB-2015}, so we have included both scaling.

We structure the paper as follows. Section 2 is devoted to technical results which will be necessary in Section 3 to prove the main result, Theorem \ref{th-MH}, as well as  Propositions \ref{othcas}-\ref{pro-qnrs}.  Finally, in Section 4 we discuss the consequences of the Mehler-Heine formula on the asymptotic behavior of the zeros of these $q$--hypergeometric polynomials. We illustrate this discussion with a variety of numerical examples and we leave some questions open, concretely one about the zeros of the $q$--function $z^{1-\alpha}J_{\alpha-1}^{(2)}\left(z(1-q) ; q\right). $

\section{Technical results} \label{sec2}

This section is devoted to obtaining some properties about the $q$-Pochhammer symbol defined in (\ref{q-Pochhammer}) and the $q$-Gamma function given by (\ref{q-gamma}). Actually, we are going to establish six technical statements which are indispensable to prove our main result in the following section. We do not claim that these results are new, but we have not found anything like them in our search in the literature.

\begin{lemma} Let $k$ be a nonnegative integer and $z$  a complex number,  such that $\Gamma_{q}(z)$ is  well defined, then
\begin{equation}\label{q-gamma-relation}
\frac{\Gamma_{q}(z+k)}{\Gamma_{q}(z)}=\frac{\left(q^{z} ; q\right)_k}{(1-q)^{k}}.
\end{equation}
\end{lemma}

\begin{proof}
Using (\ref{q-Pochhammer}) and (\ref{q-gamma-recurrence}) it follows that
\begin{equation*}
\frac{\Gamma_{q}(z+k)}{\Gamma_{q}(z)}=[z]_q[z+1]_q\dots[z+k-1]_q= \frac{\left(q^{z} ; q\right)_{k}}{(1-q)^{k}}.
\end{equation*}
\end{proof}

\begin{lemma} Let $a$ be  a positive real number and $b$ a complex number.  Then, we have, for any $k\in\mathbb{Z}$ fixed,
\begin{equation}\label{q-stirling}
\lim _{n \rightarrow+\infty} \frac{\Gamma_{q}(an+b+k)}{\Gamma_{q}(an+b)}[an+b]_{q}^{-k}=1.
\end{equation}
\end{lemma}

\begin{proof}
First, we prove the result for  a nonnegative integer $k$. Then, using (\ref{q-gamma-recurrence}) in a recursive way, we get
\begin{align*}
\frac{\Gamma_{q}(an+b+k)}{\Gamma_{q}(an+b)}[an+b]_{q}^{-k}
=\prod_{j=0}^{k-1}\frac{[an+b+j]_{q}}{[an+b]_{q}}.
\end{align*}
Since $\displaystyle\lim _{n \rightarrow+\infty} \frac{[an+b+j]_q}{[an+b]_q}=1$ for all $j\in\{0,1,\dots,k-1\}$ we get the result. When $k$ is a negative integer, we can adapt the above proof easily applying (\ref{q-gamma-recurrence}) to the denominator instead of the numerator.
\end{proof}

\begin{lemma}\label{lemma3}
Let $k$ be a nonnegative integer. Then,
\begin{equation}\label{limit-qn}
\lim _{n \rightarrow+\infty} \frac{\left(q^{-n} ; q\right)_{k}}{[n]_{q}^{k} q^{-nk}}=(-1)^{k} q^{\binom{k}{2}}(1-q)^{k}.
\end{equation}
\end{lemma}

\begin{proof}
From \cite[f. (1.8.18)]{Koekoek-book-hyper}, we have
$$\lim _{n \rightarrow+\infty} \frac{\left(q^{-n} ; q\right)_{k}}{[n]_{q}^{k} q^{-nk}}=
(-1)^{k} q^{\binom{k}{2}}\lim _{n \rightarrow+\infty} \frac{(q ; q)_{n}}{(q ; q)_{n-k}[n]_{q}^{k}}.$$
Now, using (\ref{q-gamma-relation}) and (\ref{q-stirling}), the limit on the right side of the above expression can be computed as
\begin{equation*}
(-1)^{k}q^{\binom{k}{2}}\lim _{n \rightarrow+\infty} \frac{(q ; q)_{n}}{(q ; q)_{n-k}[n]_{q}^{k}}=(-1)^{k} q^{\binom{k}{2}}(1-q)^{k} \lim _{n \rightarrow+\infty} \frac{\Gamma_{q}(n+1) [n]_q^{-k}}{\Gamma_{q}(n-k+1)}=(-1)^{k} q^{\binom{k}{2}}(1-q)^{k},
\end{equation*}
from where the result arises.
\end{proof}

\begin{lemma}\label{lemma4}
Let $b$ be a complex number. Then, for a positive real number $a$, it holds
\begin{equation}\label{limit-qanb}
\lim _{n \rightarrow+\infty} \frac{\left(q^{an+b} ; q\right)_{k}}{(1-q)^{k}[n]_{q^{a}}^{k}}=[a]_{q}^{k}.
\end{equation}
\end{lemma}

\begin{proof} For $a>0$ and $0<q<1$ we deduce easily
$\displaystyle\lim _{n \rightarrow+\infty} \frac{[n]_q}{[n]_{q^a}}=[a]_q$ and $\displaystyle\lim _{n \rightarrow+\infty} \frac{[a n+b]_q}{[n]_q}=1$. Then, using (\ref{q-gamma-relation}), (\ref{q-stirling}) and the above limits the result follows.
\end{proof}

For the next results, we assume $\binom{i}{j}=0$ when $i<j.$

\begin{lemma}\label{lemma5} We have for $n\geq1$,
\begin{equation}\label{inequality1}\left|\frac{\left(q^{-n} ; q\right)_{k}}{[n]_{q}^{k} q^{-nk}}\right| \leq q^{\binom{k}{2}}, \quad k=0,1, \ldots, n.\end{equation}
\end{lemma}
\begin{proof}
For $k=0$ the proof is trivial and the equality is reached. For $k\ge 1,$ using  \cite[f. (1.8.18)]{Koekoek-book-hyper}, we have
\begin{eqnarray*}
\left| \frac{\left(q^{-n} ; q\right)_{k}}{[n]_{q}^{k} q^{-nk}}\right|
=\left|  \frac{(q ; q)_{n}}{(q ; q)_{n-k}} \frac{(-1)^{k} q^{\binom{k}{2}-kn}}{[n]_{q}^{k}\ q^{-k n}}\right|=q^{\binom{k}{2}}  \frac{(q ; q)_{n}}{(q ; q)_{n-k} [n]_{q}^{k}}=q^{\binom{k}{2}} \left(\frac{1-q}{1-q^n}\right)^k  \prod_{j=n-k}^{n-1}\left(1-q^{j+1}\right).
\end{eqnarray*}

\noindent
Clearly,  $\displaystyle\left(\frac{1-q}{1-q^{n}}\right)^{k}\leq1$  and $\displaystyle \prod_{j=n-k}^{n-1}\left(1-q^{j+1}\right)<1.$  Then,
$$
q^{\binom{k}{2}} \left(\frac{1-q}{1-q^n}\right)^k  \prod_{j=n-k}^{n-1}\left(1-q^{j+1}\right)\leq q^{\binom{k}{2}}, $$
which proves the result.
\end{proof}

In the next two statements we provide useful bounds for
$$
\left|\frac{\left(q^{an+b} ; q\right)_{k}}{(1-q)^{k}[n]_{q^{a}}^{k}}\right|, \quad k=0,1,\ldots, n.
$$
where $b$ is a complex number with some restrictions (see Proposition \ref{lower-bound}) and $a$ is a positive real number. These bounds allow us to prove the main result of this paper.

\begin{prop}\label{upper-bound}
Let $a$ be a positive real number and $b=\gamma+i\beta$ a complex number. Then,  we have for $n\geq1,$
$$\left|\frac{\left(q^{an+b} ; q\right)_{k}}{(1-q)^{k}[n]_{q^{a}}^{k}}\right| \leq
\frac{[a]_{q}^{k}}{(1-q^a)^{k}} \left(\sqrt{1+q^{2a+2\gamma}+2q^{a+\gamma}}\right)^k, \quad k=0,1,\ldots, n. $$
\end{prop}

\begin{proof} We notice that  $1> q^a \geq q^{an}>0$ for $n\ge 1,$ then
$$
0<\frac{1}{1-q^{an}}\leq \frac{1}{1-q^{a}}, \qquad n\ge 1.
$$
We are going to use  (\ref{q-Number}) and this inequality  to prove the result, so

\begin{align*}
&\left|\frac{\left(q^{an+b} ; q\right)_{k}}{(1-q)^{k} [n]_{q^{a}}^{k}}\right|
=\frac{[a]_{q}^{k}}{\left(1-q^{a n}\right)^{k}} \prod_{j=0}^{k-1}\left|1-q^{a n+\gamma+j+i\beta}\right| \\
\leq&\frac{[a]_{q}^{k}}{\left(1-q^{a}\right)^{k}} \prod_{j=0}^{k-1}\left|1-q^{a n+\gamma+j+i\beta}\right|
=\frac{[a]_{q}^{k}}{\left(1-q^{a}\right)^{k}} \prod_{j=0}^{k-1}\left|1-q^{a n+\gamma+j}e^{i\beta \ln(q)}\right| \\
=&\frac{[a]_{q}^{k}}{\left(1-q^{a}\right)^{k}} \prod_{j=0}^{k-1}\sqrt{1+q^{2(a n+\gamma+j)}-2q^{a n+\gamma+j}\cos(\beta\ln(q))} \\
\leq&\frac{[a]_{q}^{k}}{\left(1-q^{a}\right)^{k}} \prod_{j=0}^{k-1}\sqrt{1+q^{2(a n+\gamma+j)}+2q^{a n+\gamma+j}}
\leq\frac{[a]_{q}^{k}}{\left(1-q^{a}\right)^{k}} \left(\sqrt{1+q^{2(a+\gamma)}+2q^{a+\gamma}}\right)^k.
\end{align*}
\end{proof}

Next, we use the notation $\mathbb{Z}_{-}$ for the set formed by the number $0$ and the negative integers, i.e. $\mathbb{Z}_{-}=\{0,-1,-2, \ldots \}.$

\begin{prop}\label{lower-bound}
Let $a$ be a positive real number and $b=\gamma+i\beta$ a complex number.  We assume that $an+\gamma \notin \mathbb{Z}_{-}$  for all $n$ positive integer. Then,  it exists  $\varepsilon>0$ such that for $n\ge 1$ and $k=0,1,\ldots, n$, we have
$$\left|\frac{\left(q^{an+b} ; q\right)_{k}}{(1-q)^{k}[n]_{q^{a}}^{k}}\right| \geq\left\{\begin{array}{ll}
{[a]_{q}^{k}}>0 & \text { if } \gamma \geq 0, \\ & \\
{[a]_{q}^{k} \varepsilon^{k}}>0 & \text { if } \gamma<0.
\end{array}\right.$$
\end{prop}

\begin{proof}
To prove this lower bound, we will use the well--known inequality $|z-w|\geq\big||z|-|w| \big|$ where $z$ and $w$ are complex numbers, and the equality
$$\left|q^{a n+\gamma+j+i\beta}\right|=q^{a n+\gamma+j}.$$
Then, we have
\begin{align*}
\left|\frac{\left(q^{an+b} ; q\right)_{k}}{(1-q)^{k}[n]_{q^{a}}^{k}}\right|
&=\frac{[a]_{q}^{k}}{\left(1-q^{an}\right)^{k}} \prod_{j=0}^{k-1}\left|1-q^{a n+\gamma+j+i\beta}\right|
=[a]_{q}^{k} \prod_{j=0}^{k-1} \frac{\left|1-q^{an+\gamma+j+i\beta} \right|}{1-q^{a n}}\\
&\ge [a]_{q}^{k} \prod_{j=0}^{k-1} \frac{\left|1-\left|q^{an+\gamma+j+i\beta} \right| \right|}{1-q^{a n}} =[a]_{q}^{k} \prod_{j=0}^{k-1} \frac{\left|1-q^{an+\gamma+j} \right|}{1-q^{a n}}.
\end{align*}

Now, we distinguish two cases:
\begin{itemize}
  \item When $\gamma\geq0$, taking into account $q^{an}\geq q^{an+\gamma+j}$, we obtain that
  $$
  \prod_{j=0}^{k-1} \frac{1-q^{an+\gamma+j}}{1-q^{a n}}\ge 1,$$ and the result follows.

\item When $\gamma<0$, we have
\begin{equation*}
[a]_{q}^{k} \prod_{j=0}^{k-1} \frac{\left|1-q^{an+\gamma+j} \right|}{1-q^{a n}}
=\frac{[a]_{q}^{k}}{q^{-k\gamma}} \prod_{j=0}^{k-1} \frac{\left|q^{-\gamma}-q^{an+j} \right|}{1-q^{a n}}
\geq[a]_{q}^{k} \prod_{j=0}^{k-1} \left|q^{-\gamma}-q^{an+j} \right|.
\end{equation*}

Now, on the one hand, assuming $a$ and $b$ fixed, it exits a positive integer $n_0$ that depends on $\gamma$ such that for all $n>n_0$, we have
$$q^{-\gamma}>\frac{q^{-\gamma}}{2}>q^{an}\geq q^{an+j}>0,$$
with $j\in\{0,1,\dots,k-1\}$. So, we get
$$0<\frac{q^{-\gamma}}{2}=q^{-\gamma}-\frac{q^{-\gamma}}{2}<q^{-\gamma}-q^{an}\leq q^{-\gamma}-q^{an+j}.$$
Thus, we can affirm that for $n> n_0,$
$$[a]_{q}^{k} \prod_{j=0}^{k-1} \left|q^{-\gamma}-q^{an+j} \right|
\geq[a]_{q}^{k} \frac{q^{-k\gamma}}{2^k}.$$

On the other hand, for $n\in\{1,2,\dots,n_0\},$ $k=0, \ldots, n$, and assuming  $an+\gamma \notin \mathbb{Z}_{-}$  with $n$ any positive integer, we have that  $-\gamma \neq a n+j$ for any $j$ nonnegative integer. Thus, we define
\begin{align*}
\Delta=\left\{1,|q^{-\gamma}-q^{2a}|,|q^{-\gamma}-q^{2a+1}|,|q^{-\gamma}-q^{3a}|,|q^{-\gamma}-q^{3a+1}|,
|q^{-\gamma}-q^{3a+2}|,\dots,|q^{-\gamma}-q^{n_0a+n_0-1}|\right\}.
\end{align*}
Taking $\delta:=\min\Delta>0$, we obtain
$$[a]_{q}^{k} \prod_{j=0}^{k-1} \left|q^{-\gamma}-q^{an+j} \right|
\geq[a]_{q}^{k} \delta^k.$$
Finally, if we define $\varepsilon:=\min\{\frac{q^{-\gamma}}{2},\delta\}>0$ the result holds for $n\geq1$.
\end{itemize}

\end{proof}

\begin{nota}
Under the assumptions posed in the previous propositions, we have  an upper bound and a lower bound for the quantity
$\left|\frac{\left(q^{an+b} ; q\right)_{k}}{(1-q)^{k}[n]_{q^{a}}^{k}}\right|.$ For our purposes, and without loss of generality, we can affirm that there are two constants, $\mathfrak{C}_a$ and $\mathfrak{D}_a$, independent of n, so that
\begin{equation}\label{inequality2}0<\mathfrak{C}^k_a\leq\left|\frac{\left(q^{an+b} ; q\right)_{k}}{(1-q)^{k}[n]_{q^{a}}^{k}}\right|\leq \mathfrak{D}^k_a.\end{equation}
\end{nota}

\section{Main result}

In this section we obtain the main goal of this paper: the  Mehler--Heine asymptotics for some $q$-hypergeometric polynomials. Before stating this result,  we still
have to take a further step on this issue giving a relation for the $q$-Bessel function  (\ref{bessel-q-2}).

\begin{prop} \label{prop-bqg}
Let $\alpha\in \mathbb{R}\backslash \mathbb{Z}_{-}$ be. Then,
\begin{equation*}\label{q-bessel-relation}
\Gamma_{q}(\alpha)z^{\frac{1-\alpha}{2}} J_{\alpha-1}^{(2)}(2 \sqrt{z}(1-q) ; q)
=\ _0\phi_{1}\left(\begin{array}{c}
                                                                                                                                             - \\
                                                                                                                                               q^{\alpha}
                                                                                                                                             \end{array}; q,-z(q-1)^{2} q^{\alpha}\right).
\end{equation*}
\end{prop}

\begin{proof}
From (\ref{bessel-q-2}), we have
\begin{eqnarray*}
J_{\alpha-1}^{(2)}(x ; q)&=&\frac{\left(q^{\alpha} ; q\right)_{\infty}}{(q ; q)_{\infty}}\left(\frac{x}{2}\right)^{\alpha-1} \ _0\phi_{1}\left(\begin{array}{c}
                                                                                                                                               - \\
                                                                                                                                               q^{\alpha}
                                                                                                                                             \end{array}
 ; q, \frac{-q^{\alpha}x^{2}}{4}\right)\\
&=&\frac{\left(q^{\alpha} ; q\right)_{\infty}}{\left(q; q\right)_{\infty}}\left(\frac{x}{2}\right)^{\alpha-1} \sum_{k=0}^{\infty} q^{2\binom{k}{2}}\frac{(-1)^{k} x^{2 k} q^{\alpha k}}{4^k\left(q^{\alpha} ; q\right)_{k}(q ; q)_{k}}.
\end{eqnarray*}
Then, making the change $x^2=4z$ and introducing the factor $(1-q)^{1-\alpha}$,  we have the following identity,
$$\frac{(q ; q)_{\infty}}{\left(q^{\alpha} ; q\right)_{\infty}}(1-q)^{1-\alpha} z^{\frac{1-\alpha}{2}} J_{\alpha-1}^{(2)}(2 \sqrt{z} ; q)=
(1-q)^{1-\alpha} \sum_{k=0}^{\infty} \frac{q^{2\binom{k}{2}}(-1)^{k} z^{k}q^{\alpha k}}{\left(q^{\alpha} ; q\right)_{k}(q ; q)_{k}}. $$
Now, using (\ref{q-basic}) and (\ref{q-gamma}), we get
$$\Gamma_q(\alpha) z^{\frac{1-\alpha}{2}} J_{\alpha-1}^{(2)}(2 \sqrt{z} ; q)=
(1-q)^{1-\alpha}\  _0\phi_{1}\left(\begin{array}{l}
- \\
q^{\alpha}
\end{array} ; q, -zq^{\alpha}\right). $$
Finally, it is enough to make the change of variable $\sqrt{z}\rightarrow\sqrt{z}(1-q)$  to end the proof.
\end{proof}

We have all the ingredients to establish our main result.

\begin{theo}[\textbf{Mehler--Heine asymptotics}] \label{th-MH}
We use the notation from (\ref{notation1}-\ref{notation3}), assuming that $\alpha\in \mathbb{R}\backslash \mathbb{Z}_{-}$, $a_j>0$, $c_{j}>0$ and that $b_j$ and $d_{j}$ are complex numbers satisfying $a_j n+\mathrm{Re}(b_j) \notin \mathbb{Z}_{-}$ and $c_{j}n+\mathrm{Re}(d_{j}) \notin \mathbb{Z}_{-}$  with $j\in\{1,2,\dots,s-1\}$ and $s\geq2$. Then,
\begin{align}\nonumber
&\lim_{n\to+\infty} \ _s\phi_{s}\left(\begin{array}{l}
q^{-n},q^{\mathfrak{a}_sn+\mathfrak{b}_s} \\
q^{\alpha},q^{\mathfrak{c}_sn+\mathfrak{d}_s}
\end{array} ; q, \frac{q^{n+\alpha}[n]_{q^{\mathfrak{c}_s}}}{[n]_q[n]_{q^{\mathfrak{a}_s}}}(q-1)z\right)\\ \label{MH}
&= \left( \frac{[\mathfrak{a}_s]_q}{[\mathfrak{c}_s]_q}z\right)^{\frac{1-\alpha}{2}} \Gamma_{q}(\alpha) J_{\alpha-1}^{(2)}\left(2 (1-q) \sqrt{\frac{[\mathfrak{a}_s]_q}{[\mathfrak{c}_s]_q}z} ; q\right),
\end{align}
uniformly on compact subsets of the complex plane.  For $s=1,$ it holds
\begin{equation}\label{case1}
\lim_{n\to+\infty} \ _1\phi_{1}\left(\begin{array}{l}
q^{-n} \\
q^{\alpha},
\end{array} ; q, \frac{q^{n+\alpha}}{[n]_q}(q-1)z\right)= z^{\frac{1-\alpha}{2}} \Gamma_{q}(\alpha) J_{\alpha-1}^{(2)}\left(2 (1-q) \sqrt{z} ; q\right).
\end{equation}
\end{theo}

\begin{proof} First, we observe that the quantities $a_jn+b_j $ and $c_{j}n+d_{j}$, for $j=1, \ldots, s-1$ and $s\geq2$, satisfy the hypothesis  posed in Proposition \ref{lower-bound}.

Now, scaling the variable $z$ in (\ref{q-basic}) in the following way $\displaystyle z\to \frac{q^{n+\alpha}[n]_{q^{\mathfrak{c}_s}}}{[n]_q[n]_{q^{\mathfrak{a}_s}}}(q-1)z$,  we get
\begin{align*}
&\ _s\phi_{s}\left(\begin{array}{l}
q^{-n},q^{\mathfrak{a}_sn+\mathfrak{b}_s} \\
q^{\alpha},q^{\mathfrak{c}_sn+\mathfrak{d}_s}
\end{array} ; q, \frac{q^{n+\alpha}[n]_{q^{\mathfrak{c}_s}}}{[n]_q[n]_{q^{\mathfrak{a}_s}}}(q-1)z\right)\\
&=\sum_{k=0}^{n}\frac{\left(q^{-n},q^{\mathfrak{a}_sn+\mathfrak{b}_s};q\right)_k}
{\left(q^{\alpha},q^{\mathfrak{c}_sn+\mathfrak{d}_s};q\right)_k}
(-1)^k q^{\binom{k}{2}}\frac{ q^{(n+\alpha) k}[n]_{q^{\mathfrak{c}_s}}^k (q-1)^kz^k}
{[n]_q^k[n]_{q^{\mathfrak{a}_s}}^k(q;q)_k}\\
&:=\sum_{k=0}^{n} g_{n,k}(z).
\end{align*}

Using (\ref{limit-qn}-\ref{limit-qanb}), we have for $k$ fixed
\begin{equation} \label{limit-proof-MH}
\lim_{n\to+\infty}g_{n,k}(z)= q^{2\binom{k}{2}} (1-q)^k  \frac{[\mathfrak{a}_s]_{q}^k}{[\mathfrak{c}_s]_{q}^k}
\frac{ z^k q^{\alpha k} (-1)^k (1-q)^k}{\left(q^{\alpha};q\right)_k\left(q;q\right)_k},
\end{equation}
uniformly on compact subsets of the complex plane. Furthermore, we take $z$ on a compact subset $\Omega$ of the complex plane, so $|z|\le \mathfrak{C}_{\Omega}.$ Then, for $n\ge 0 $ and  $0\le k\le n,$ we get
\begin{equation}\label{leb}
\left| g_{n,k}(z) \right|
\leq\frac{ \mathfrak{D}_{a_1}^k \mathfrak{D}_{a_2}^k \cdots \mathfrak{D}_{a_{s-1}}^k }
{ \mathfrak{C}_{c_1}^k \mathfrak{C}_{c_2}^k \cdots \mathfrak{C}_{c_{s-1}}^k }
\frac{q^{2\binom{k}{2}} q^{\alpha k} (1-q)^k}{\left|\left(q^{\alpha};q\right)_k\right|(q;q)_k} \mathfrak{C}_{\Omega}^k  : =g_k(z),
\end{equation}
where we have used (\ref{inequality1}) and (\ref{inequality2}). Thus, we have found a dominant for  $\sum_{k=0}^{n} g_{n,k}(z).$ This dominant  is convergent, i.e. the series $
\sum_{k=0}^{+\infty}g_k(z)$ 
converges. We can see this by applying the  D'Alembert (or quotient) criterion for series of nonnegative terms.

Now,  to apply the Lebesgue's dominated convergence theorem, for $n,k\geq0$ we define $\mathcal{F}_{n,k}(z)$ as
$$\mathcal{F}_{n,k}(z):=\left\{
                          \begin{array}{ll}
                            g_{n,k}(z), & \hbox{if $0\leq k \leq n$;} \\
                            0, & \hbox{if $n<k$.}
                          \end{array}
                        \right.
$$
Then, using (\ref{limit-proof-MH}) and (\ref{leb}) and taking  $z$ on a compact subset $\Omega$ of the complex plane, we have  for each $k$
\begin{align*}
\lim_{n\to+\infty}\mathcal{F}_{n,k}(z)&= q^{2\binom{k}{2}} (1-q)^k  \frac{[\mathfrak{a}_s]_{q}^k}{[\mathfrak{c}_s]_{q}^k}
\frac{ z^k q^{\alpha k} (-1)^k (1-q)^k}{\left(q^{\alpha};q\right)_k\left(q;q\right)_k},\\
\left|\mathcal{F}_{n,k}(z)\right|&\le g_k(z).
\end{align*}

Thus, we can write
$$
\sum_{k=0}^{n} g_{n,k}(z)= \int  \mathcal{F}_{n,k}(z) d\mu(k),
$$
where $d\mu(k)$ is the discrete measure with support on the nonnegative integers ($k=0,1,\dots$) and with a mass equal to one in each point of the support. Then, we apply the Lebesgue's dominated convergence theorem and (\ref{limit-proof-MH}), obtaining
\begin{align*}
&\lim_{n\to+\infty}\sum_{k=0}^{n}\frac{\left(q^{-n},q^{\mathfrak{a}_sn+\mathfrak{b}_s};q\right)_k}
{\left(q^{\alpha},q^{\mathfrak{c}_sn+\mathfrak{d}_s};q\right)_k}
(-1)^k q^{\binom{k}{2}}\frac{ q^{(n+\alpha) k}[n]_{q^{\mathfrak{c}_s}}^k (q-1)^kz^k}
{[n]_q^k[n]_{q^{\mathfrak{a}_s}}^k(q;q)_k}
= \lim_{n\to+\infty}\sum_{k=0}^{n} g_{n,k}(z)\\
=&\lim_{n\to+\infty}\int  \mathcal{F}_{n,k}(z) d\mu(k)= \int \lim_{n\to+\infty}  \mathcal{F}_{n,k}(z) d\mu(k)
=\sum_{k=0}^{+\infty}(-1)^{k} q^{2\binom{k}{2}}  \frac{[\mathfrak{a}_s]_{q}^k}{[\mathfrak{c}_s]_{q}^k}
\frac{ z^k q^{\alpha k}  (1-q)^{2k}}{\left(q^{\alpha};q\right)_k\left(q;q\right)_k}\\
=&\ _0\phi_{1}\left(\begin{array}{c}
     - \\
         q^{\alpha}
  \end{array}; q,-\frac{z[\mathfrak{a}_s]_{q}^k}{[\mathfrak{c}_s]_{q}^k}(q-1)^{2} q^{\alpha}\right).
\end{align*}
Finally, using  Proposition \ref{prop-bqg} we get (\ref{MH}).  The  proof of (\ref{case1}) is similar, but now it is not necessary to use either the limit (\ref{lemma4}) or the bounds (\ref{inequality2}).
\end{proof}

\medskip

Now, we can tackle the case $r-1\leq s.$ As we have commented previously, in this case, as far as we know, the limit function cannot be expressed as a known $q$--hypergeometric function except when $r=s$, then we get the same result as in Theorem \ref{th-MH}.

\begin{prop}\label{othcas}
We take $r\geq 1$, $s\geq1$, $r-1\leq s$, and  $\alpha\in \mathbb{R}\backslash \mathbb{Z}_{-}.$   We consider $b_j$  and $d_{\ell}$ complex numbers satisfying $a_j n+\mathrm{Re}(b_j) \notin \mathbb{Z}_{-},$  $c_{\ell}n+\mathrm{Re}(d_{\ell}) \notin \mathbb{Z}_{-}$ where $a_j>0$, $c_{\ell}>0$ with $j\in\{1,2,\dots,r-1\}$ and $\ell\in\{1,2,\dots,s-1\}.$ Then, for $r\geq2$ and $s\geq2$,
\begin{align*}&\lim_{n\to+\infty} \ _r\phi_{s}\left(\begin{array}{l}
q^{-n},q^{\mathfrak{a}_rn+\mathfrak{b}_r} \\
q^{\alpha},q^{\mathfrak{c}_sn+\mathfrak{d}_s}
\end{array} ; q, \frac{q^{n+\alpha}[n]_{q^{\mathfrak{c}_s}}}{[n]_q [n]_q{^{\mathfrak{a}_r}}}(q-1)z\right)\\
=&\sum_{k=0}^{+\infty}(-1)^{(1+r-s)k} q^{(2+s-r)\binom{k}{2}}  \frac{[\mathfrak{a}_r]_q^k}{[\mathfrak{c}_s]_q^k }
\frac{ z^k q^{\alpha k}  (1-q)^{(2+r-s)k}}{\left(q^{\alpha};q\right)_k\left(q;q\right)_k}.
\end{align*}
In addition,
\begin{align}
\label{p4-1}\lim_{n\to+\infty} \ _1\phi_{s}\left(\begin{array}{l}
q^{-n} \\
q^{\alpha},q^{\mathfrak{c}_sn+\mathfrak{d}_s}
\end{array} ; q, \frac{q^{n+\alpha}[n]_q^{\mathfrak{c}_s}}{[n]_q}(q-1)z\right)&=\sum_{k=0}^{+\infty}(-1)^{sk} q^{(s+1)\binom{k}{2}}
\frac{ z^k q^{\alpha k}  (1-q)^{(3-s)k}}{[\mathfrak{c}_s]_q^k\left(q^{\alpha};q\right)_k\left(q;q\right)_k},\\
\label{p4-2}
\lim_{n\to+\infty} \ _r\phi_{1}\left(\begin{array}{l}
q^{-n},q^{\mathfrak{a}_rn+\mathfrak{b}_r} \\
q^{\alpha}
\end{array} ; q, \frac{q^{n+\alpha}}{[n]_q [n]_q{^{\mathfrak{a}_r}}}(q-1)z\right)
&=\sum_{k=0}^{+\infty}(-1)^{rk} q^{(3-r)\binom{k}{2}}  [\mathfrak{a}_r]_q^k
\frac{ z^k q^{\alpha k}  (1-q)^{(r+1)k}}{\left(q^{\alpha};q\right)_k\left(q;q\right)_k}.
\end{align}
\end{prop}
\begin{proof}
Observe that when $r\geq2$ and $s\geq2$, we can write
\begin{align*}
& \ _r\phi_{s}\left(\begin{array}{l}
q^{-n},q^{\mathfrak{a}_rn+\mathfrak{b}_r} \\
q^{\alpha},q^{\mathfrak{c}_sn+\mathfrak{d}_s}
\end{array} ; q, \frac{q^{n+\alpha}[n]_{q^{\mathfrak{c}_s}}}{[n]_q [n]_q{^{\mathfrak{a}_r}}}(q-1)z\right)\\
=&\sum_{k=0}^{n}\frac{\left(q^{-n},q^{\mathfrak{a}_rn+\mathfrak{b}_r};q\right)_k}
{\left(q^{\alpha},q^{\mathfrak{c}_sn+\mathfrak{d}_s};q\right)_k}
(-1)^{(1+s-r)k} q^{(1+s-r)\binom{k}{2}}\frac{ q^{(n+\alpha) k}[n]_{q^{\mathfrak{c}_s}}^k (q-1)^kz^k}{[n]_q^k [n]_{q^{\mathfrak{a}_s}}^k (q;q)_k}\\
:=&\sum_{k=0}^{n} g_{n,k}^{[r,s]}(z).
\end{align*}
Moreover, we can prove for $k$ fixed that
\begin{align*}
\lim_{n\to+\infty}g_{n,k}^{[r,s]}(z)=
(-1)^{(1+r-s)k} q^{(2+s-r)\binom{k}{2}}  \frac{[\mathfrak{a}_r]_q^k}{[\mathfrak{c}_s]_q^k}
\frac{q^{\alpha k}  (1-q)^{(2+r-s)k}z^k}{\left(q^{\alpha};q\right)_k\left(q;q\right)_k}.
\end{align*}
uniformly on compact subsets of the complex plane. Then, acting like in the proof of Theorem \ref{th-MH} we can apply the Lebesgue's dominated convergence theorem, getting the result.  Finally, the asymptotic relations (\ref{p4-1}) and (\ref{p4-2}) can be obtained in the same way by handling  the notation adequately.
\end{proof}

\begin{nota} \label{ocrms}
It is worth noting that the condition $r-1\le s$ is necessary to apply the Lebesgue's dominated convergence theorem in the two previous proofs.
\end{nota}

\begin{nota}
We can observe that Theorem \ref{th-MH} recovers partially one result for hypergeometric polynomials ${ }_{r} F_{s}\left(\begin{array}{l}
a_{1}, \ldots, a_{r} \\
b_{1}, \ldots, b_{s}
\end{array} ; z\right)$ given in  \cite[Th. 1]{Bracciali-JJMB-2015}. This occurs when $r=s.$ To see this, it is enough to consider (\ref{limqnum}) and  notice that
when $q\to1$ we have
$$\lim_{q\to1}\frac{q^{n+\alpha}[n]_q^{\mathfrak{c}_s}z}{[n]_q [n]_q^{\mathfrak{a}_s}}=\frac{z}{n}. $$
Then, using (\ref{relationH-qH-r=s}), (\ref{q-gamma-limit}) and (\ref{bessel-q-2}-\ref{q-bessel-limit})    in Theorem \ref{th-MH}, we deduce Theorem 1 in \cite{Bracciali-JJMB-2015} when $r=s.$
\end{nota}

As we have mentioned previously in the introduction, one of the referees  proposed us to use the scaling $z\to q^nz$. We have obtained the following statements.
\begin{prop}\label{qn}
We use  the notation from (\ref{notation1}-\ref{notation3}), assuming that $\alpha\in \mathbb{R}\backslash \mathbb{Z}_{-}$, $a_j>0$, $c_{j}>0$ and that $b_j$ and $d_{j}$ are complex numbers satisfying $a_j n+\mathrm{Re}(b_j) \notin \mathbb{Z}_{-}$ and $c_{j}n+\mathrm{Re}(d_{j}) \notin \mathbb{Z}_{-}$  with $j\in\{1,2,\dots,s-1\}$ and $s\geq2$. Then,
\begin{align*}
\lim_{n\to+\infty} \ _s\phi_{s}\left(\begin{array}{l}
q^{-n},q^{\mathfrak{a}_sn+\mathfrak{b}_s} \\
q^{\alpha},q^{\mathfrak{c}_{s}n+\mathfrak{d}_{s}}
\end{array} ; q, q^nz\right)
&=\ _0\phi_{1}\left(\begin{array}{c}
     - \\
         q^{\alpha}
  \end{array}; q,z\right)\\
&= \left( \frac{-z}{(q-1)^2q^{\alpha}}\right)^{\frac{1-\alpha}{2}} \Gamma_{q}(\alpha) J_{\alpha-1}^{(2)}\left(2  \sqrt{\frac{-z}{q^{\alpha}}} ; q\right).
\end{align*}
For $s=1$, we have
\begin{equation}\label{p5-1}
\lim_{n\to+\infty} \ _1\phi_{1}\left(\begin{array}{l}
q^{-n}\\
q^{\alpha}
\end{array} ; q, q^nz\right)
=\ _0\phi_{1}\left(\begin{array}{c}
     - \\
         q^{\alpha}
  \end{array}; q,z\right)= \left( \frac{-z}{(q-1)^2q^{\alpha}}\right)^{\frac{1-\alpha}{2}} \Gamma_{q}(\alpha) J_{\alpha-1}^{(2)}\left(2  \sqrt{\frac{-z}{q^{\alpha}}} ; q\right),
\end{equation}
uniformly on compact subsets of the complex plane.
\end{prop}

\begin{proof}
We need to modify appropriately some of the results from Section  \ref{sec2}. From Lemmas \ref{lemma3}-\ref{lemma5} we can deduce
\begin{align}\label{qn-1}
\lim _{n \rightarrow+\infty} q^{nk}\left(q^{-n} ; q\right)_{k}&=(-1)^{k} q^{\binom{k}{2}},\\
\label{qn-2}
\lim _{n \rightarrow+\infty} \left(q^{an+b} ; q\right)_{k}&=1,\\
\label{qn-3}
\left| q^{nk}\left(q^{-n} ; q\right)_{k}\right|\leq q^{\binom{k}{2}}, &\quad k=0,1, \ldots, n.
\end{align}
In addition, under the assumptions posed in Propositions \ref{upper-bound} and \ref{lower-bound},  we use the same technique to establish that there are two constants, $\widehat{\mathfrak{C}_a}$ and $\widehat{\mathfrak{D}_a}$, independent of $n$, satisfying
\begin{equation}\label{qn-4}0<\widehat{\mathfrak{C}^k_a}\leq\left|\left(q^{an+b} ; q\right)_{k}\right|\leq \widehat{\mathfrak{D}^k_a}.\end{equation}
From (\ref{qn-1}-\ref{qn-4}) we can prove the result in the same way as in Theorem \ref{th-MH}.  The  proof of (\ref{p5-1}) is similar, but now it is not necessary to use either the limit (\ref{qn-2}) or the bounds (\ref{qn-4}).
\end{proof}

\begin{nota}
Notice that the result in Proposition \ref{qn} does not depend on $s$. This is due to the type of scaling and to the fact that all $q$-numbers disappear in the limits (\ref{qn-1}-\ref{qn-2}).
\end{nota}

 Now, we discuss the case $r-1\leq s$ making the scaling $z\to q^nz$.

\begin{prop} \label{pro-qnrs}
We take $r\geq1$, $s\geq1$, $r-1\leq s$, and  $\alpha\in \mathbb{R}\backslash \mathbb{Z}_{-}.$   We consider $b_j$  and $d_{\ell}$ complex numbers satisfying $a_j n+\mathrm{Re}(b_j) \notin \mathbb{Z}_{-},$  $c_{\ell}n+\mathrm{Re}(d_{\ell}) \notin \mathbb{Z}_{-}$ where $a_j>0$, $c_{\ell}>0$ with $j\in\{1,2,\dots,r-1\}$ and $\ell\in\{1,2,\dots,s-1\}.$ Then,  for $r\geq2$ and $s\geq2$,
\begin{align*}&\lim_{n\to+\infty} \ _r\phi_{s}\left(\begin{array}{l}
q^{-n},q^{\mathfrak{a}_rn+\mathfrak{b}_r} \\
q^{\alpha},q^{\mathfrak{c}_sn+\mathfrak{d}_s}
\end{array} ; q, q^nz\right)
=\sum_{k=0}^{+\infty}
\frac{(-1)^{(r-s)k} q^{(2+s-r)\binom{k}{2}}}{\left(q^{\alpha};q\right)_k\left(q;q\right)_k}z^k.
\end{align*}
In addition,
\begin{align*}
\lim_{n\to+\infty} \ _1\phi_{s}\left(\begin{array}{l}
q^{-n} \\
q^{\alpha},q^{\mathfrak{c}_sn+\mathfrak{d}_s}
\end{array} ; q, q^nz\right)
&=\sum_{k=0}^{+\infty}
\frac{(-1)^{(1-s)k} q^{(s+1)\binom{k}{2}}}{\left(q^{\alpha};q\right)_k\left(q;q\right)_k}z^k,\\
\lim_{n\to+\infty} \ _r\phi_{1}\left(\begin{array}{l}
q^{-n},q^{\mathfrak{a}_rn+\mathfrak{b}_r} \\
q^{\alpha}
\end{array} ; q, q^nz\right)
&=\sum_{k=0}^{+\infty}
\frac{(-1)^{(r-1)k} q^{(3-r)\binom{k}{2}}}{\left(q^{\alpha};q\right)_k\left(q;q\right)_k}z^k.\\
\end{align*}
\end{prop}

\begin{proof}
The proof is totally similar to the one of Proposition \ref{othcas} but in this case we use (\ref{qn-1}-\ref{qn-4}).
\end{proof}

\subsection{Two classical examples}

Now, we use Theorem \ref{th-MH} and Proposition \ref{othcas} to obtain the Mehler--Heine formula for two important families of basic hypergeometric polynomials.

\begin{ejemplo}
Using this theorem we can obtain a well-known type of asymptotics for the $q$-Laguerre orthogonal polynomials given by Moak \cite[Theorem 5]{Moak-1981}, although the author uses another standardization for the polynomials. A generalization of Moak's result  was given by Ismail in  \cite[Theorem 21.8.4]{Ismail-2005} where he provided an asymptotic expansion.  These $q$-Laguerre orthogonal polynomials, which we denote by $L_n^{(\alpha)}(x;q)$, are orthogonal with respect to (see \cite[f. (21.8.4)]{Ismail-2005})
\begin{equation*}
\int_{0}^{\infty} L_{m}^{(\alpha)}(z ; q) L_{n}^{(\alpha)}(z ; q) \frac{z^{\alpha} d x}{(-z ; q)_{\infty}}=-\frac{\pi}{\sin (\pi \alpha)} \frac{\left(q^{-\alpha} ; q\right)_{\infty}}{(q ; q)_{\infty}} \frac{\left(q^{\alpha+1} ; q\right)_{n}}{q^{n}(q ; q)_{n}} \delta_{m, n},
\end{equation*}
with $\alpha>-1$ and if $\alpha=k$ with $k=0,1,2,\dots$ the right-hand side  is interpreted as
$$\left(\ln q^{-1}\right)q^{-\binom{k+1}{2}-n}\left(q^{n+1};q\right)_k \delta_{m,n}.$$

These polynomials can be written as (see \cite[f. (21.8.2)]{Ismail-2005})
\begin{equation*}\label{q-laguerre-Ismail}
L_{n}^{(\alpha)}(z ; q)=\frac{\left(q^{\alpha+1} ; q\right)_{n}}{(q ; q)_{n}} \sum_{k=0}^{n} \frac{\left(q^{-n} ; q\right)_{k}}{(q ; q)_{k}} q^{\binom{k+1}{2}} \frac{z^{k} q^{(\alpha+n) k}}{\left(q^{\alpha+1} ; q\right)_{k}}.
\end{equation*}

After some simple algebraic computations we obtain
\begin{equation*}\label{q-laguerre-koekoek}
L_{n}^{(\alpha)}(z ; q)=\frac{\left(q^{\alpha+1} ; q\right)_{n}}{(q ; q)_{n}} \ _{1}\phi_{1}\left(\begin{array}{c}
q^{-n} \\
q^{\alpha+1}
\end{array} ; q,-q^{n+\alpha+1} z\right)
\end{equation*}
Now, scaling the variable adequately and applying Theorem 1 and (\ref{q-gamma}), we get

\begin{align*}
&\lim_{n\to+\infty}L_{n}^{(\alpha)}\left(\frac{z}{(1-q)[n]_q}; q\right)
=\lim_{n\to+\infty}\frac{\left(q^{\alpha+1} ; q\right)_{n}}{(q ; q)_{n}} \ _{1}\phi_{1}\left(\begin{array}{c}
q^{-n} \\
q^{\alpha+1}
\end{array} ; q,\frac{-q^{n+\alpha+1} z}{(1-q)[n]_q}\right)\\
&=\frac{\left(q^{\alpha+1} ; q\right)_{\infty}}{(q ; q)_{\infty}}2^{\alpha}\Gamma_q(\alpha+1)\left(2\sqrt{\frac{z}{(q-1)^2}}\right)^{-\alpha}
J_{\alpha}^{(2)}\left(2\sqrt{\frac{z}{(q-1)^2}}(1-q);q\right)\\
&=\frac{\left(q^{\alpha+1} ; q\right)_{\infty}}{(q ; q)_{\infty}}\Gamma_q(\alpha+1)z^{-\alpha/2}(1-q)^{\alpha}
J_{\alpha}^{(2)}\left(2\sqrt{z};q\right)=z^{-\alpha/2}J_{\alpha}^{(2)}\left(2\sqrt{z};q\right),
\end{align*}

But, we know by Moak's (or Ismail's) result that
$$\lim_{n\to+\infty}L_{n}^{(\alpha)}\left(z; q\right)=z^{-\alpha/2}J_{\alpha}^{(2)}\left(2\sqrt{z};q\right),$$
uniformly on compact subsets of the complex plane.
Then, we have
\begin{equation*}
\label{moak}
\lim_{n\to+\infty}L_{n}^{(\alpha)}(z ; q)=\lim_{n\to+\infty}L_{n}^{(\alpha)}\left(\frac{z}{(1-q)[n]_q}; q\right)=z^{-\alpha/2}J_{\alpha}^{(2)}\left(2\sqrt{z};q\right),
\end{equation*}
uniformly on compact subsets of the complex plane.

In fact, the first equality in the above expression is expected, for example, adapting conveniently the proof of Corollary 1 in \cite{ambpr}. But, as we have seen, it is also easy to obtain it using Theorem \ref{th-MH}.
\end{ejemplo}

\medskip

\begin{ejemplo}
The Little $q$-Jacobi orthogonal polynomials are usually denoted by $p_n(z;a,b|q)$. For $0<aq<1$ and $bq<1$ these polynomials are orthogonal with respect to (see \cite[f. 14.12.2]{Koekoek-book-hyper})
$$
\sum_{k=0}^{\infty} \frac{(b q ; q)_{k}}{(q ; q)_{k}}(a q)^{k} p_{m}\left(q^{k} ; a, b | q\right) p_{n}\left(q^{k} ; a, b |q\right)
=\frac{\left(a b q^{2} ; q\right)_{\infty}}{(a q ; q)_{\infty}} \frac{(1-a b q)(a q)^{n}}{\left(1-a b q^{2 n+1}\right)} \frac{(q, b q ; q)_{n}}{(a q, a b q ; q)_{n}} \delta_{m n}.
$$
These polynomials have a basic hypergeometric representation given by (see \cite[f. 14.12.1]{Koekoek-book-hyper})
$$p_n(z;a,b|q)=\ _2\phi_{1}\left(\begin{array}{l}
q^{-n},abq^{n+1} \\
aq
\end{array} ; q, qz\right)=\ _2\phi_{1}\left(\begin{array}{l}
q^{-n},q^{n+1+\ln(ab)/\ln(q)} \\
q^{1+\ln(a)/\ln(q)}
\end{array} ; q, qz\right),$$
where we have used
\begin{equation}\label{cqn}
cq^{an+b}=q^{an+b+\ln(c)/\ln(q)},
\end{equation}
for $c>0$.
Then, using (\ref{cqn}), Proposition \ref{othcas} and scaling the variable as
$$\displaystyle qz\to \frac{q^{n+1+\ln(a)/\ln(q)}}{[n]_q^2}(q-1)z,$$ we obtain
$$\lim_{n\to+\infty}p_n\left(\frac{q^{n+1+\ln(a)/\ln(q)}}{[n]_q^2}(q-1)z;a,b|q\right)=
\sum_{k=0}^{+\infty} q^{\binom{k}{2}}
\frac{z^k q^{(1+\ln(a)/\ln(q)) k}  (1-q)^{3k}}{\left(q^{1+\ln(a)/\ln(q)};q\right)_k\left(q;q\right)_k},$$
assuming  the condition $1+\ln(a)/\ln(q)>-1$.
\end{ejemplo}

\section{Discussion about the zeros}

Throughout this section we have assumed that the assumptions of Theorem \ref{th-MH} hold.  The values $b_j$ and $d_j$ in (\ref{MH}) can be complex numbers so the $q-$hypergeometric polynomial (\ref{q-basic})
\begin{equation*}\label{poly}
\ _s\phi_{s}\left(\begin{array}{l}
q^{-n},q^{\mathfrak{a}_sn+\mathfrak{b}_s} \\
q^{\alpha},q^{\mathfrak{c}_sn+\mathfrak{d}_s}
\end{array} ; q, z(q-1)\right),
\end{equation*}
is a polynomial of degree $n$ with complex coefficients, and all its zeros can be nonreal complex. We enumerate them by $x_{q,n,k}$ with $1\leq k\leq n$, but obviously they do not have to be ordered. Taking this into consideration together with  Hurwitz’s Theorem (see \cite[Th. 1.91.3]{sz}), then there exits a limit relation between the scaled zeros of these $q$-hypergeometric polynomials and the zeros of the limit function in Theorem \ref{th-MH}. Concretely, the scaled zeros
\begin{equation} \label{scaled-zeros}
x_{q,n,k}^{\ast}:= \frac{[n]_q[n]_{q^{\mathfrak{a}_s}}}{q^{n+\alpha}[n]_{q^{\mathfrak{c}_s}}}x_{q,n,k},
\end{equation}
converge to the zeros of the function
\begin{equation}\label{limit-function}
\left( \frac{[\mathfrak{a}_s]_q}{[\mathfrak{c}_s]_q}z\right)^{\frac{1-\alpha}{2}} J_{\alpha-1}^{(2)}\left(2 (1-q) \sqrt{\frac{[\mathfrak{a}_s]_q}{[\mathfrak{c}_s]_q}z} ; q\right),
\end{equation}
when $n\to+\infty$ that we will denote by $z_{\ell}$. Moreover, when $\alpha>0$ the zeros of (\ref{limit-function}) are real and simple (see \cite[Th. 4.2]{Ismail-1982}), so the zeros $x_{q,n,k}^{\ast}$, given by (\ref{scaled-zeros}), converge to real numbers.

However, as far as we know, when $\alpha<0$ ($\alpha \notin \mathbb{Z}_{-}$)  there are not any results about the zeros of the function $z^{\frac{1-\alpha}{2}}J_{\alpha-1}^{(2)}\left(2 \sqrt{z}(1-q) ; q\right)$ analogous to the ones for the function $z^{\frac{1-\alpha}{2}}J_{\alpha-1}(2\sqrt{z})$ given, for example, in \cite[Pag. 483-484]{Watson-1922}. Thus, we have not been able to establish a result for the zeros as detailed as in Proposition 1 in \cite{Bracciali-JJMB-2015}.

Following the analysis of \cite[Remark 3]{Bracciali-JJMB-2015} (see also \cite[Pag. 483-484]{Watson-1922}), we know that the number of nonreal zeros of $z^{\frac{1-\alpha}{2}}J_{\alpha-1}(2\sqrt{z})$ depend on the value of $\alpha$.  In fact, $w^{1-\alpha}J_{\alpha-1}(w)$ has at most two purely imaginary zeros. Taking into account the change of variable $w=2\sqrt{z}$, these two   purely imaginary zeros are transformed into a negative real zero of function $z^{\frac{1-\alpha}{2}}J_{\alpha-1}(2\sqrt{z})$. Furthermore, we know exactly when these purely imaginary zeros appear.

In our framework, taking into account the numerical experiments that we have made, we believe  that the number of nonreal  zeros of $w^{1-\alpha}J_{\alpha-1}^{(2)}\left(w (1-q); q\right)$ also depends on $\alpha$ and coincides with the number of nonreal  zeros  of $w^{1-\alpha}J_{\alpha-1}(w)$. However, we feel that there is a relevant fact: the number of purely imaginary zeros can be greater than two and depends on the value of $q$ too. In this way, we observe that when $q\to1$ the number of purely imaginary zeros of the function $w^{1-\alpha}J_{\alpha-1}^{(2)}\left(w(1-q) ; q\right)$ is the same as the one of $w^{1-\alpha}J_{\alpha-1}(w)$. We are posing these comments as a conjecture.

To illustrate the previous conjecture  we are going to show some numerical experiments. We have used the symbolic computer program \ma \, language \textit{12.1.1} (also known as Wolfram language) to make them.

We take the following data:
$$s=3, \quad q=1/2, $$ and

\begin{center}
\begin{tabular}{|l|l|l|l|}
  \hline
  $a_1=3$ & $b_1=6$  & $c_1=4/3$ & $d_1=2-3i$ \\ \hline
  $a_2=5/4$ & $b_2=-2/3+2i$    & $c_2=5/6$   & $d_2=1$ \\ \hline
\end{tabular}
\end{center}
With this choice the $q$--hypergeometric polynomials have complex coefficients.  We denote a zero of the limit function (\ref{limit-function}) by  $z_{\ell}.$

\medskip
\noindent \textbf{First experiment.}
We take $\alpha=1$. In this situation, using (see \cite[Th. 4.2]{Ismail-1982}), all the zeros of (\ref{limit-function}) are real and simple.   In the next tables we show the behavior of the scaled zeros converging to the  first three positive zeros of (\ref{limit-function}).

\begin{center}
\begin{tabular}{cl}
     & $x_{1/2,n,1}^{\ast}$   \\ \hline
  $n=10$  & $1.192838457109 - 0.000373946559i$   \\
  $n=20$  & $1.191325673237 - 6.242164358333\times10^{-8}i$     \\
  $n=40$  & $1.191320585494 - 1.806507624485\times10^{-15}i$     \\ \bottomrule
  $z_{\ell}$ & $1.191320585443$  \\ \bottomrule
\end{tabular}
\end{center}

\begin{center}
\begin{tabular}{cl}
     & $x_{1/2,n,2}^{\ast}$   \\ \hline
  $n=10$  & $10.5677725749753 - 0.001544756947i$     \\
  $n=20$  & $10.5384014795907 - 2.574926397462\times10^{-7}i$     \\
  $n=40$  & $10.5383205870082 - 7.451913839285\times10^{-15}i$     \\ \bottomrule
  $z_{\ell}$ & $10.5383205862745$     \\ \bottomrule
\end{tabular}
\end{center}

\begin{center}
\begin{tabular}{cl}
     & $x_{1/2,n,3}^{\ast}$   \\ \hline
  $n=10$  & $54.673251810109 - 0.003501823800i$     \\
  $n=20$  & $54.420541178838 - 5.826491777755\times10^{-7}i$     \\
  $n=40$  & $54.419973744204 - 1.686199155483\times10^{-15}i$     \\ \bottomrule
  $z_{\ell}$ & $54.419973739632$     \\ \bottomrule
\end{tabular}
\end{center}

\medskip

\noindent \textbf{Second experiment.}
Let's take $\alpha=-51/100.$ With this value the function $w^{1-\alpha}J_{\alpha-1}(w)$ has  two  purely imaginary zeros and the rest of them are positive real zeros. This also occurs for the function $w^{1-\alpha}J_{\alpha-1}^{(2)}\left(w(1-q) ; q\right)$. Thus,   the limit function (\ref{limit-function}) has a negative zero and the rest of them  are positive real zeros. We only show the convergence to the first two positive zeros and to the negative zero of (\ref{limit-function}).

\begin{center}
\begin{tabular}{cl}
     & $x_{1/2,n,1}^{\ast}$   \\ \hline
  $n=10$  & $-0.257907814869 + 0.000069310549i$   \\
  $n=20$  & $-0.257444571262 + 1.156063483173\times10^{-8}i$     \\
  $n=40$  & $-0.257443205893 + 3.345689303361\times10^{-16}i$     \\ \bottomrule
  $z_{\ell}$ & $-0.257443205880$  \\ \bottomrule
\end{tabular}
\end{center}

\begin{center}
\begin{tabular}{cl}
    & $x_{1/2,n,2}^{\ast}$   \\ \hline
  $n=10$  & $1.717581648483 - 0.000304499740 i$     \\
  $n=20$  & $1.712978988208 - 5.076635738362\times10^{-8}i$     \\
  $n=40$  & $1.712966620706 - 1.469194883628\times10^{-15}i$     \\ \bottomrule
  $z_{\ell}$ & $1.712966620595$     \\ \bottomrule
\end{tabular}
\end{center}

\begin{center}
\begin{tabular}{cl}
   & $x_{1/2,n,3}^{\ast}$   \\ \hline
  $n=10$  & $15.162942911386 - 0.001127913056i$     \\
  $n=20$  & $15.093385289371 - 1.877379036189\times10^{-7}i$     \\
  $n=40$  & $15.093230236139 - 5.433179181139\times10^{-15}i$     \\ \bottomrule
  $z_{\ell}$ & $15.093230234896$     \\ \bottomrule
\end{tabular}
\end{center}

\medskip

\noindent \textbf{Third experiment.}
We take $\alpha=-78/10$. Then, the function $w^{1-\alpha}J_{\alpha-1}\left(w\right)$ has 16 nonreal zeros, but they are not  purely imaginary zeros, and the rest of them are positive real zeros. However, in our case the function $w^{1-\alpha}J_{\alpha-1}^{(2)}\left(w(1-q) ; q\right)$ has 16 nonreal zeros,  8 of which are  purely imaginary zeros being the rest of them positive real zeros. As we have previously commented in the conjecture, the number of purely imaginary zeros of the function $w^{1-\alpha}J_{\alpha-1}^{(2)}\left(w(1-q) ; q\right)$  can be greater than 2 according to the value of $q$. Therefore, the limit function (\ref{limit-function}) has 4 nonreal zeros and 4 negative real zeros being the rest of them positive.

\begin{center}
\begin{tabular}{cl}
   & $x_{1/2,n,1}^{\ast}$   \\ \hline
  $n=10$  & $-26.583782080249 + 13.144830793624i$   \\
  $n=20$  & $-24.217517881849 + 13.495110134590i$     \\
  $n=40$  & $-24.215268341967 + 13.495244741513i$     \\ \bottomrule
  $z_{\ell}$ & $-24.215268337780 + 13.495244740537i$  \\ \bottomrule
\end{tabular}
\end{center}

\begin{center}
\begin{tabular}{cl}
   & $x_{1/2,n,2}^{\ast}$   \\ \hline
  $n=10$  & $-26.584069012013 - 13.145175170701i$   \\
  $n=20$  & $-24.217517928700 - 13.495110193829i$     \\
  $n=40$  & $-24.215268341967 - 13.495244741517i$     \\ \bottomrule
  $z_{\ell}$ & $-24.215268337780 - 13.495244740538i$  \\ \bottomrule
\end{tabular}
\end{center}

\begin{center}
\begin{tabular}{cl}
   & $x_{1/2,n,3}^{\ast}$   \\ \hline
  $n=10$  & $-11.012428251874 + 62.145264518016i$   \\
  $n=20$  & $-6.977745534060  + 56.644723043027i$     \\
  $n=40$  & $-6.974490887747  + 56.639531189402i$     \\ \bottomrule
  $z_{\ell}$ & $-6.974490884028 + 56.639531179756i$  \\ \bottomrule
\end{tabular}
\end{center}

\begin{center}
\begin{tabular}{cl}
   & $x_{1/2,n,4}^{\ast}$   \\ \hline
  $n=10$  & $-11.012731503979 - 62.145767769738i$   \\
  $n=20$  & $-6.977745580595 - 56.644723126264i$     \\
  $n=40$  & $-6.974490887747 - 56.639531189402i$     \\ \bottomrule
  $z_{\ell}$ & $-6.974490884028 - 56.639531179756i$  \\ \bottomrule
\end{tabular}
\end{center}

\begin{center}
\begin{tabular}{cl}
   & $x_{1/2,n,5}^{\ast}$   \\ \hline
  $n=10$  & $-16.166116185151 - 0.00004242253i$   \\
  $n=20$  & $-15.859752369736 - 9.93118142529\times10^{-9}i$     \\
  $n=40$  & $-15.859327530746 - 2.87493401651\times10^{-16}i$     \\ \bottomrule
  $z_{\ell}$ & $-15.859327529012$  \\ \bottomrule
\end{tabular}
\end{center}

\begin{center}
\begin{tabular}{cl}
   & $x_{1/2,n,6}^{\ast}$   \\ \hline
  $n=10$  & $-8.372365289781 + 1.061870351232\times10^{-6}i$   \\
  $n=20$  & $-8.358589490584 + 2.877523388912\times10^{-10}i$     \\
  $n=40$  & $-8.358521100270 + 8.331309933918\times10^{-18}i$     \\ \bottomrule
  $z_{\ell}$ & $-8.358521099510$  \\ \bottomrule
\end{tabular}
\end{center}

\begin{center}
\begin{tabular}{cl}
   & $x_{1/2,n,7}^{\ast}$   \\ \hline
  $n=10$  & $-4.183493772518 + 7.184651279381\times10^{-6}i$   \\
  $n=20$  & $-4.174952644245 + 1.169874676784\times10^{-12}i$     \\
  $n=40$  & $-4.174916567642 - 8.016528868055\times10^{-20}i$     \\ \bottomrule
  $z_{\ell}$ & $-4.174916567260$  \\ \bottomrule
\end{tabular}
\end{center}

\begin{center}
\begin{tabular}{cl}
   & $x_{1/2,n,8}^{\ast}$   \\ \hline
  $n=10$  & $-2.085524807895 + 0.001212554069i$   \\
  $n=20$  & $-2.087470339683 + 2.036177391751\times10^{-7}i$     \\
  $n=40$  & $-2.087472421386 + 5.892871393618\times10^{-15}i$     \\ \bottomrule
  $z_{\ell}$ & $-2.087472421388$  \\ \bottomrule
\end{tabular}
\end{center}

\begin{center}
\begin{tabular}{cl}
  & $x_{1/2,n,9}^{\ast}$   \\ \hline
  $n=10$  & $222.192716349085 - 0.000476699235i$   \\
  $n=20$  & $185.503465533781 - 7.245543670015\times10^{-8}i$     \\
  $n=40$  & $185.473370922148 - 2.096712467864\times10^{-15}i$     \\ \bottomrule
  $z_{\ell}$ & $185.473370878074$  \\ \bottomrule
\end{tabular}
\end{center}

\medskip

\noindent \textbf{Fourth Experiment.}
We take $\alpha=-88/10$. Then, the function $w^{1-\alpha}J_{\alpha-1}\left(w\right)$ has 18 nonreal zeros, 2 of which  are purely imaginary zeros, and the rest of them are positive real zeros. Again, the function $w^{1-\alpha}J_{\alpha-1}^{(2)}\left(w(1-q) ; q\right)$ also has 18 nonreal zeros, but now  10 of them  are  purely imaginary zeros, being the rest of them  positive real zeros. Thus, the limit function (\ref{limit-function}) has 4 nonreal zeros and 5 negative real zeros being the rest of them positive.

\begin{center}
\begin{tabular}{cl}
  & $x_{1/2,n,1}^{\ast}$   \\ \hline
  $n=10$  & $-58.731778291761 + 24.983339723902i$   \\
  $n=20$  & $-48.142375334033 + 27.211453956792i$     \\
  $n=40$  & $-48.133792760941 + 27.212205588614i$     \\ \bottomrule
  $z_{\ell}$ & $-48.133792748725 + 27.212205587084i$  \\ \bottomrule
\end{tabular}
\end{center}

\begin{center}
\begin{tabular}{cl}
   & $x_{1/2,n,2}^{\ast}$   \\ \hline
  $n=10$  & $-58.732065276581 - 24.983660249396i$   \\
  $n=20$  & $-48.142375378613 - 27.211454014981i$     \\
  $n=40$  & $-48.133792760941 - 27.212205588614i$     \\ \bottomrule
  $z_{\ell}$ & $-48.133792748725 - 27.212205587084i$  \\ \bottomrule
\end{tabular}
\end{center}

\begin{center}
\begin{tabular}{cl}
   & $x_{1/2,n,3}^{\ast}$   \\ \hline
  $n=10$  & $-34.303303302277 + 138.270486284223i$   \\
  $n=20$  & $-13.528504443799 + 113.526606000274i$     \\
  $n=40$  & $-13.515626962528 + 113.506768620178i$     \\ \bottomrule
  $z_{\ell}$ & $-13.515626949090 + 113.506768591833i$  \\ \bottomrule
\end{tabular}
\end{center}

\begin{center}
\begin{tabular}{cl}
   & $x_{1/2,n,4}^{\ast}$   \\ \hline
  $n=10$  & $-34.303631075141 - 138.270994273937i$   \\
  $n=20$  & $-13.528504488949 - 113.526606084040i$     \\
  $n=40$  & $-13.515626962528 - 113.506768620178i$     \\ \bottomrule
  $z_{\ell}$ & $-13.515626949090 - 113.506768591833i$  \\ \bottomrule
\end{tabular}
\end{center}

\begin{center}
\begin{tabular}{cl}
   & $x_{1/2,n,5}^{\ast}$   \\ \hline
  $n=10$  & $-32.743951055138 - 0.000026360939i$   \\
  $n=20$  & $-31.625976675655 - 9.688310475475\times10^{-9}i$     \\
  $n=40$  & $-31.624501885216 - 2.805403876599\times10^{-16}i$     \\ \bottomrule
  $z_{\ell}$ & $-31.624501881169$  \\ \bottomrule
\end{tabular}
\end{center}

\begin{center}
\begin{tabular}{cl}
  & $x_{1/2,n,6}^{\ast}$   \\ \hline
  $n=10$  & $-16.740893534529 + 5.765133760916\times10^{-7}i$   \\
  $n=20$  & $-16.719718514606 + 2.835085312132\times10^{-10}i$     \\
  $n=40$  & $-16.719591520682 + 8.212134980550\times10^{-18}i$     \\ \bottomrule
  $z_{\ell}$ & $-16.719591519169$  \\ \bottomrule
\end{tabular}
\end{center}

\begin{center}
\begin{tabular}{cl}
   & $x_{1/2,n,7}^{\ast}$   \\ \hline
  $n=10$  & $-8.367018219885 - 8.764399202657\times10^{-10}i$   \\
  $n=20$  & $-8.349886270829 - 2.470953096732\times10^{-12}i$     \\
  $n=40$  & $-8.349814077147 - 7.157189188727\times10^{-20}i$     \\ \bottomrule
  $z_{\ell}$ & $-8.349814076382$  \\ \bottomrule
\end{tabular}
\end{center}

\begin{center}
\begin{tabular}{cl}
   & $x_{1/2,n,8}^{\ast}$   \\ \hline
  $n=10$  & $-4.183511394310 + 7.147780631497\times10^{-6}i$   \\
  $n=20$  & $-4.174980970526 + 3.925233400802\times10^{-12}i$     \\
  $n=40$  & $-4.174944906183 + 3.561296499329\times10^{-22}i$     \\ \bottomrule
  $z_{\ell}$ & $-4.174944905801$  \\ \bottomrule
\end{tabular}
\end{center}

\begin{center}
\begin{tabular}{cl}
  & $x_{1/2,n,9}^{\ast}$   \\ \hline
  $n=10$  & $-2.085538601937 + 0.001209882024i$   \\
  $n=20$  & $-2.087470352958 + 2.031659087739\times10^{-7}i$     \\
  $n=40$  & $-2.087472390014 + 5.879794684903\times10^{-15}i$     \\ \bottomrule
  $z_{\ell}$ & $-2.087472390016$  \\ \bottomrule
\end{tabular}
\end{center}

\begin{center}
\begin{tabular}{cl}
  & $x_{1/2,n,10}^{\ast}$   \\ \hline
  $n=10$  & $568.386160039012 - 0.000543294101i$   \\
  $n=20$  & $371.816049857173 - 7.279580302175\times10^{-8}i$     \\
  $n=40$  & $371.698707738027 - 2.106396160012\times10^{-15}i$     \\ \bottomrule
  $z_{\ell}$ & $371.698707595280$  \\ \bottomrule
\end{tabular}
\end{center}

\noindent \textbf{Fifth Experiment.}
We consider the same data as in the third experiment, but changing the value of $q$. Thus, we take $q=9/10.$  Now, we can observe that the function  $w^{1-\alpha}J_{\alpha-1}^{(2)}\left(w(1-q) ; q\right)$ has 16 nonreal zeros
  but it has not any purely imaginary zeros, so the limit function (\ref{limit-function}) has not any negative real zeros. In fact, it has 8 nonreal zeros and the rest of them are positive numbers. In the next tables we can also notice that the convergence is slower than in the other experiments. Thus, we have observed in all the numerical experiments made that when $q$ approaches to $1$ the convergence slows down.

  To avoid including many tables, we provide the values of  8 nonreal zeros of (\ref{limit-function}) and two examples. These zeros are

\begin{align*}
  -2.899672548545819 &\pm 0.866690762364998i, &\quad -2.054705381959291 &\pm 2.480492842076364i, \\
  -0.280308180558639 &\pm 3.667024750741143i, &\quad  2.632871686846202 &\pm 3.803425181595050i.
\end{align*}

\begin{center}
\begin{tabular}{cl}
  & $x_{9/10,n,7}^{\ast}$   \\ \hline
  $n=20$       & $3.163297460938 + 4.534270108808i$     \\
  $n=40$       & $2.688577412718 + 3.889727638001i$     \\ \bottomrule
  $z_{\ell}$ & $2.632871686846 + 3.803425181595i$  \\ \bottomrule
\end{tabular}
\end{center}

\begin{center}
\begin{tabular}{cl}
  & $x_{9/10,n,8}^{\ast}$   \\ \hline
  $n=20$       & $2.983050648402 - 4.679512324145i$     \\
  $n=40$       & $2.678631275315 - 3.897696238845i$     \\ \bottomrule
  $z_{\ell}$ & $2.632871686846 - 3.803425181595i$  \\ \bottomrule
\end{tabular}
\end{center}

\begin{center}
\begin{tabular}{cl}
   & $x_{9/10,n,9}^{\ast}$   \\ \hline
  $n=20$       & $9.365586442180 - 0.176724181442i$     \\
  $n=40$       & $7.857805351203 - 0.009728505494i$     \\ \bottomrule
  $z_{\ell}$ & $7.668210419914$  \\ \bottomrule
\end{tabular}
\end{center}

The third and fourth experiments show the differences pointed out in the conjecture about the number of purely imaginary zeros of the functions  $w^{1-\alpha}J_{\alpha-1}\left(w\right)$ and $w^{1-\alpha}J_{\alpha-1}^{(2)}\left(w(1-q) ; q\right)$ depending on the value of $\alpha.$ The third and fifth experiments also show the existence of purely imaginary zeros depending on $q$ too.

In conclusion, to obtain more asymptotic properties about the zeros of the $q$--hypergeometric polynomials
\begin{equation*}
\ _s\phi_{s}\left(\begin{array}{l}
q^{-n},q^{\mathfrak{a}_sn+\mathfrak{b}_s} \\
q^{\alpha},q^{\mathfrak{c}_sn+\mathfrak{d}_s}
\end{array} ; q, z(q-1)\right),
\end{equation*}
we need to go further in the knowledge of the zeros of the function $w^{1-\alpha}J_{\alpha-1}^{(2)}\left(w(1-q) ; q\right). $ That question remains open as far as we know.

\section*{Acknowledgments}
We thank the referees sincerely for their careful revision of the paper. Their suggestions and comments have been very useful to improve it. The authors  are partially supported by the Ministry of Science, Innovation, and
Universities of Spain and the European Regional Development Fund (ERDF), grant MTM2017-89941-P; they are
also partially supported by ERDF and Consejer\'{\i}a de Econom\'{\i}a, Conocimiento, Empresas y Universidad de la
Junta de Andaluc\'{\i}a (grant UAL18-FQM-B025-A) and by Research Group FQM-0229 (belonging to Campus of
International Excellence CEIMAR). The author J.J.M-B. is also partially supported by the research centre CDTIME
of Universidad de Almer\'{\i}a and by Junta de Andaluc\'{\i}a and ERDF, Ref. SOMM17/6105/UGR.

\end{document}